\documentclass[a4paper,11pt,reqno]{amsart}
\usepackage[T1]{fontenc}
\usepackage[utf8]{inputenc}
\usepackage{lmodern}
\usepackage[text={400pt,660pt},centering]{geometry}
\usepackage{amsmath,amssymb,amsthm,mathtools}
\usepackage{graphicx,microtype,booktabs,enumitem}
\usepackage{float}
\usepackage{etoolbox}
\makeatletter
\patchcmd{\@setauthors}{\MakeUppercase}{}{}{}
\makeatother
\usepackage[colorlinks=true,linkcolor=blue,citecolor=blue,urlcolor=blue]{hyperref}
\hypersetup{pdftitle={Explosion versus decay for boundary derivatives of p-harmonic functions as p tends to 1: nonlocality},pdfauthor={Yuval Peres and Han Wang},pdfkeywords={p-Laplacian, boundary derivatives, least gradient}}
\newtheorem{theorem}{Theorem}[section]
\newtheorem{lemma}[theorem]{Lemma}

\newtheorem{corollary}[theorem]{Corollary}
\theoremstyle{remark}

\BeforeBeginEnvironment{theorem}{\begin{samepage}}
\AfterEndEnvironment{theorem}{\end{samepage}}
\BeforeBeginEnvironment{lemma}{\begin{samepage}}
\AfterEndEnvironment{lemma}{\end{samepage}}
\BeforeBeginEnvironment{proposition}{\begin{samepage}}
\AfterEndEnvironment{proposition}{\end{samepage}}
\numberwithin{equation}{section}
\setlist[enumerate,1]{label=(\alph*),leftmargin=*,itemsep=5pt}
\DeclareMathOperator{\dist}{dist}
\DeclareMathOperator{\diam}{diam}

\newcommand{\R}{\mathbb R}

\newcommand{\n}{\mathbf{n}}
\allowdisplaybreaks
\title[Boundary derivatives as $p$ tends to $1$]{E\MakeLowercase{xplosion versus decay for boundary derivatives}\\
\MakeLowercase{of }$p$\MakeLowercase{-harmonic functions as }$p$\MakeLowercase{ tends to }$1$\MakeLowercase{: nonlocality}}
\author[Yuval Peres and Han Wang]{Yuval Peres and Han Wang}
\address{Y. Peres, Beijing Institute of Mathematical Sciences and Applications, Beijing, China}
\email{yperes@bimsa.cn}
\address{H. Wang, Qiuzhen College, Tsinghua University, Beijing, China}
\email{wanghan21@mails.tsinghua.edu.cn}
\date{}
\begin{document}
\begin{abstract}
We consider the Dirichlet problem for the $p$-Laplacian on a bounded Lipschitz
domain $\Omega\subset\mathbb R^d$, with boundary data given by the indicator of a
closed set. At a smooth boundary point where the data vanish, we study the upper
and lower inward normal Dini derivatives as $p\downarrow1$. We give sufficient
geometric conditions for these derivatives to be of order $(p-1)^{-1}$, and
other conditions under which they decay exponentially in $1/(p-1)$. Whether
explosion or decay occurs is not determined locally: on one fixed planar domain
with real-analytic boundary, changing the data on a set of arbitrarily small
total length, uniformly separated from the observation point, changes explosion
into exponential decay. We also exhibit a critical example of a cylinder in
$\mathbb R^{d+1}$ where the inward normal derivative is of order
$\sqrt{d/(p-1)}$.

\end{abstract}
\maketitle

\section{Introduction}
We study the boundary behavior of $p$-harmonic functions, that is, continuous
local weak solutions of
\[
 \Delta_p u:=\nabla\cdot\bigl(|\nabla u|^{p-2}\nabla u\bigr)=0,
 \qquad 1<p<\infty.
\]
For a nonconstant $p$-harmonic function attaining its minimum at a boundary point
with an interior tangent ball, the Hopf boundary point lemma gives a positive
inward normal derivative whenever that derivative exists; see
\cite[Proposition 3.2.1]{TolksdorfHopfLemma}. Our goal is to determine the behavior
of the corresponding lower and upper Dini derivatives as $p\downarrow1$.

The variational endpoint $p=1$ naturally leads to minimizing total variation
in $BV(\Omega)$; see \cite{Juutinen}. A $BV$ function may have jumps, and
its distributional gradient may contain a singular part. Convergence to a
least-gradient function, or control of total variation, does not by itself
control the boundary normal derivatives of $p$-harmonic approximations.
These features motivate the study of normal difference quotients at the
boundary and the quantitative description of their behavior as $p\downarrow1$.

A radial example already exhibits two different scales. For $d\ge2$ and
$1<p<d$, put
\[
 \gamma=\gamma(p,d):=\frac{d-p}{p-1},\qquad
 u_p(x)=\frac{2^{-\gamma}-|x|^{-\gamma}}{2^{-\gamma}-1}.
\]
This function is $p$-harmonic on $\{1<|x|<2\}$, with value $1$ on the inner
sphere and $0$ on the outer sphere. The magnitudes of its normal derivatives are
\[
 \frac{\gamma}{1-2^{-\gamma}}
 \quad\hbox{and}\quad
 \frac{\gamma\,2^{-\gamma-1}}{1-2^{-\gamma}},
\]
respectively. The first grows as $(d-1)/(p-1)$, whereas the second is exponentially
small on the scale $(p-1)^{-1}$. See Figures~\ref{fig:radial} and
\ref{fig:radialderivatives}.
\begin{figure}[tbp]
\centering\includegraphics[width=.84\linewidth]{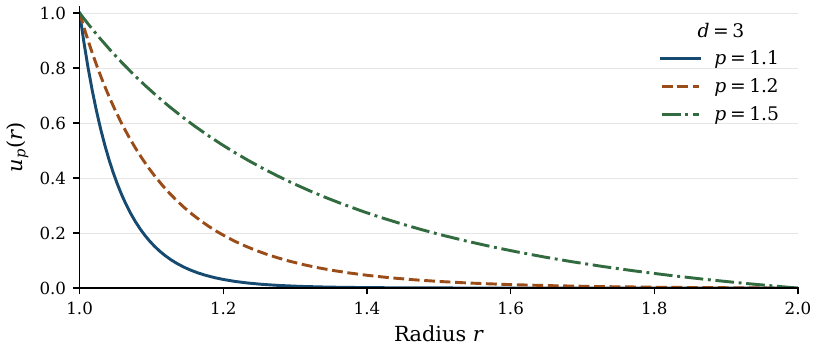}
\caption{The radial profiles $r\mapsto u_p(r)$ in dimension $d=3$ for several values of $p$.
As $p$ approaches $1$, $\lvert u_p'(1)\rvert$ diverges, whereas
$\lvert u_p'(2)\rvert$ tends to zero.}
\label{fig:radial}
\end{figure}
\begin{figure}[tbp]
\centering\includegraphics[width=\linewidth]{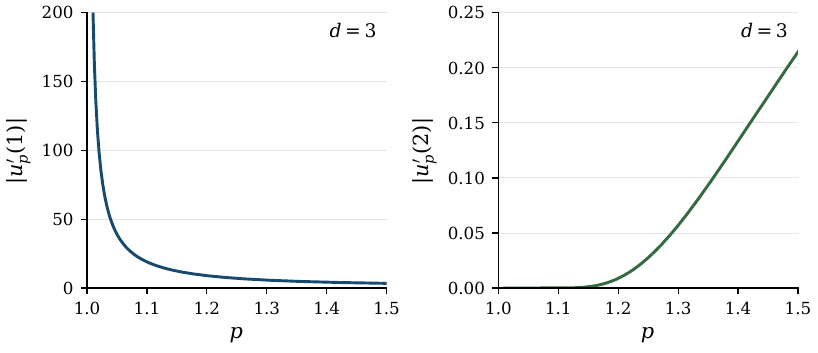}
\caption{The magnitudes of the normal derivatives on the two boundary components,
as functions of $p$, in dimension $d=3$.}
\label{fig:radialderivatives}
\end{figure}

One might expect the local curvature to determine these regimes. Our results
show that the asymptotic scale can instead change when only distant boundary
values are changed. We consider boundary values $F$ that are indicators of closed
subsets of $\partial\Omega$. On a bounded Lipschitz domain these upper
semicontinuous functions are resolutive, so the Perron solution provides a
canonical meaning of the Dirichlet problem; see Section~\ref{sec:comparison}.

Throughout the main results, $\n=\n(x_0)$ denotes the unit inward normal.
The assumptions $F(x_0)=0$ and $\{F=1\}$ closed imply that $F$ vanishes in a
boundary neighborhood of $x_0$, where the Perron solution attains the value $0$.
We use the lower and upper inward normal Dini derivatives
\begin{equation}\label{eq:normaldini}
\begin{aligned}
 \underline D_{\n}u(x_0)
 &:=\liminf_{t\downarrow0}\frac{u(x_0+t\n)-u(x_0)}{t},\\
 \overline D_{\n}u(x_0)
 &:=\limsup_{t\downarrow0}\frac{u(x_0+t\n)-u(x_0)}{t}.
\end{aligned}
\end{equation}
Our general estimates do not assume that these quantities agree. Whenever the
ordinary inward normal derivative exists, both equal that derivative.

\begin{theorem}\label{thm:explosion}
Let $d\ge2$, let $\Omega\subset\mathbb R^d$ be a bounded Lipschitz domain, and
assume that $\partial\Omega$ is $C^2$ near $x_0\in\partial\Omega$.
Let $F$ be the indicator of a closed subset of $\partial\Omega$, with $F(x_0)=0$,
and let $u_p$ be its $p$-harmonic Perron solution. For
$1<p<p_0(\Omega,F,x_0)$ the following statements hold.
\begin{enumerate}
\item There is $C=C(\Omega,F,x_0)>0$ such that
\begin{equation}\label{eq:upperderivative}
 \overline D_{\n}u_p(x_0)\le\frac{C}{p-1}.
\end{equation}
\item If there exists $R>0$ such that
\[
 \{F=0\}\setminus\{x_0\}\subset B(x_0-R\n,R),
\]
then there is $c=c(\Omega,F,x_0)>0$ such that
\begin{equation}\label{eq:lowerexplosion}
 \underline D_{\n}u_p(x_0)\ge\frac{c}{p-1}.
\end{equation}
\end{enumerate}
\end{theorem}
The ball in part~(b) is tangent at $x_0$ and its center lies in the outward-normal
direction; see Figure~\ref{fig:explosiongeometry}.
\begin{figure}[tbp]
\centering\includegraphics[width=.85\linewidth]{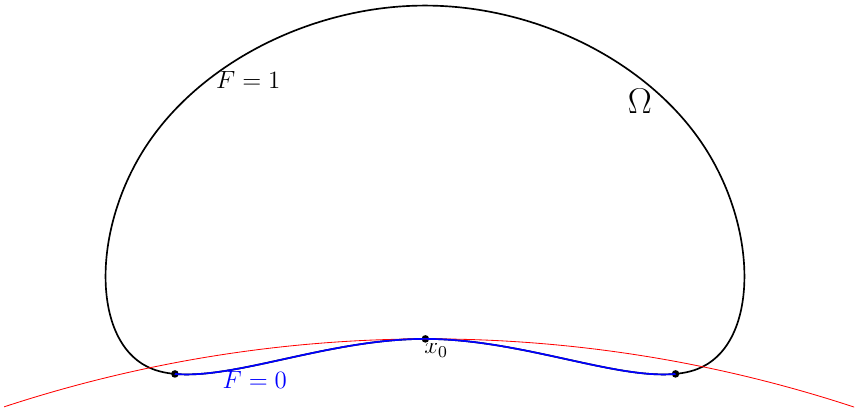}
\caption{The geometric hypothesis of Theorem~\ref{thm:explosion}(b).
The zero set, apart from $x_0$, is contained in the open ball with red boundary.}
\label{fig:explosiongeometry}
\end{figure}

\begin{theorem}\label{thm:decay}
Under the domain and boundary-value assumptions of Theorem~\ref{thm:explosion},
the following statements hold for $1<p<p_0(\Omega,F,x_0)$.
\begin{enumerate}
\item If there are $\xi\in\mathbb R^d$ and $\beta\in\mathbb R$ such that
\[
 \xi\cdot x_0>\beta,
 \qquad \{F=1\}\subset\{x:\xi\cdot x<\beta\},
\]
then there is $c=c(\Omega,F,x_0)>0$ such that
\begin{equation}\label{eq:exponentialupper}
 \overline D_{\n}u_p(x_0)\le\exp\left(-\frac{c}{p-1}\right).
\end{equation}
\item If $\{F=1\}$ has nonempty relative interior in $\partial\Omega$, then
there is $C=C(\Omega,F,x_0)>0$ such that
\begin{equation}\label{eq:exponentiallower}
 \underline D_{\n}u_p(x_0)\ge\exp\left(-\frac{C}{p-1}\right).
\end{equation}
\end{enumerate}
\end{theorem}
The hypothesis in part~(a) separates the support of $F$ from $x_0$ by a
hyperplane; see Figure~\ref{fig:decaygeometry}. The constants and the threshold
$p_0$ in these theorems may change between assertions. No uniformity in
$\Omega$, $F$, or $x_0$ is asserted.
\begin{figure}[tbp]
\centering\includegraphics[width=.85\linewidth]{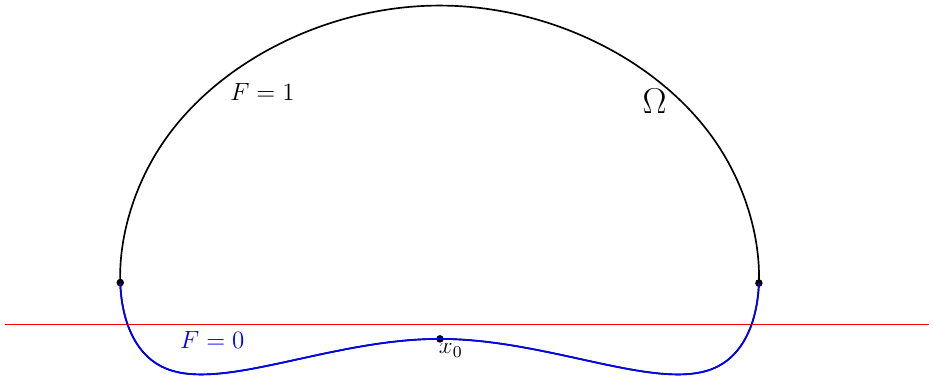}
\caption{Hyperplane separation in Theorem~\ref{thm:decay}(a) gives exponential
decay. The hypothesis does not impose a curvature condition at $x_0$.}
\label{fig:decaygeometry}
\end{figure}

Our third result concerns a critical example in which the jump in the boundary
values meets the tangent plane at the observation point.
\begin{theorem}\label{thm:critical}
Let $d\ge1$ and $Q=B^d(0,1)\times(0,1)\subset\mathbb R^{d+1}$. Let $u_p$ be
the Perron solution with boundary values
\[
 F(x,y)=\begin{cases}
 1,& |x|=1\text{ or }y=1,\\
 0,&\text{otherwise}.
 \end{cases}
\]
For $1<p<p_0(d)$,
\begin{equation}\label{eq:criticalderivative}
 \frac1{80}\sqrt{\frac{d}{p-1}}
 \le\partial_yu_p(0,0)\le2\sqrt{\frac{d}{p-1}}.
\end{equation}
\end{theorem}
\begin{figure}[tbp]
\centering\includegraphics[height=135pt]{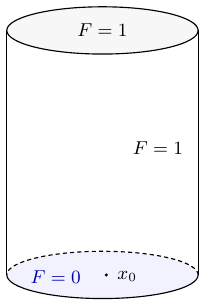}
\caption{The cylinder in Theorem~\ref{thm:critical}, with value $1$ on the top
and sides, including the bottom rim, and value $0$ on the open bottom face.}
\label{fig:critical}
\end{figure}
Thus the normal derivative can exhibit an intermediate algebraic scale.
The three theorems give sufficient conditions and a critical model example;
they do not classify all possible asymptotic behaviors.

\subsection{Nonlocality under changes on arbitrarily short boundary sets}
The two regimes can occur on one fixed domain with real-analytic boundary,
even when the boundary data differ only on a remote set of arbitrarily small
length. Fix
\[
 \Omega=\{(x,y)\in\R^2:x^2+(y+x^2-1)^2<1\},
 \qquad x_0=(0,0),\quad \n=e_2.
\]
The polynomial diffeomorphism $(x,y)\mapsto(x,y+x^2-1)$ maps $\Omega$
onto the unit disk, so $\Omega$ is bounded and its entire boundary is a
real-analytic Jordan curve. Write $q(x,y)=1-y-x^2$; on $\partial\Omega$,
\[
 x^2+q^2=1,\qquad y=q^2-q.
\]
For $0<\delta<1/2$, define closed boundary arcs by
\begin{equation}\label{eq:smallarcdata}
\begin{aligned}
 E_{1,\delta}&=\partial\Omega\cap\{q\le\delta\},\\
 E_{2,\delta}&=\partial\Omega\cap\{q\le-\delta\},
 \qquad F_{i,\delta}=\mathbf1_{E_{i,\delta}}.
\end{aligned}
\end{equation}
Both boundary functions vanish on the fixed boundary neighborhood
$\partial\Omega\cap\{q>1/2\}$ of $x_0$. They differ precisely on
\[
 S_\delta=\partial\Omega\cap\{-\delta<q\le\delta\},\qquad
 \dist(x_0,S_\delta)\ge\sqrt{1-\delta^2}\ge\frac{\sqrt3}{2}.
\]
The length of $S_\delta$ tends to zero as $\delta\downarrow0$;
see Figure~\ref{fig:smallarcnonlocal}.

\begin{figure}[H]
\centering
\includegraphics{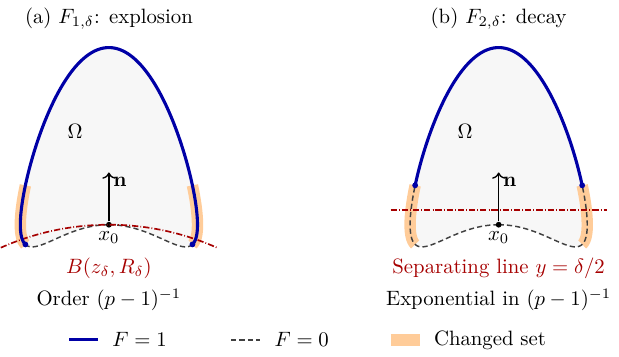}
\caption{Two boundary conditions on one fixed domain with real-analytic
boundary. They differ only on the highlighted arcs $S_\delta$, whose total
length tends to zero as $\delta\downarrow0$.
In (a), the red arc belongs to the outward tangent circle whose ball contains
every zero boundary point except $x_0$. In (b), the red line strictly
separates the support of the data from $x_0$. The drawing uses $\delta=1/3$;
the construction works for every $0<\delta<1/2$.}
\label{fig:smallarcnonlocal}
\end{figure}

Let $u_{i,\delta,p}$ be the Perron solution with boundary data $F_{i,\delta}$.
To verify the explosion condition for $F_{1,\delta}$, set
\[
 R_\delta=\delta^{-1},
 \qquad z_\delta=(0,-R_\delta).
\]
The ball $B(z_\delta,R_\delta)$ is outward tangent at $x_0$.
Every other zero of $F_{1,\delta}$ has $\delta<q<1$, and hence
\[
 \begin{aligned}
 |(x,y)-z_\delta|^2-R_\delta^2
  &=(1-q)\bigl(1+q+q^2-q^3-2q/\delta\bigr)\\
  &\le2(1-q)(1-q/\delta)<0,
 \end{aligned}
\]
where we used $1+q+q^2-q^3\le2$ for $0\le q\le1$.
Theorem~\ref{thm:explosion} therefore gives constants $a_\delta,A_\delta>0$
such that
\begin{equation}\label{eq:smallarcexplosion}
 \frac{a_\delta}{p-1}
 \le\underline D_{\n}u_{1,\delta,p}(x_0)
 \le\overline D_{\n}u_{1,\delta,p}(x_0)
 \le\frac{A_\delta}{p-1}.
\end{equation}
In contrast, on $E_{2,\delta}$ we have $q\le-\delta$ and thus
$y=q^2-q\ge\delta+\delta^2$. The line $y=\delta/2$ strictly separates
this arc from $x_0$, and the arc has nonempty relative interior.
Theorem~\ref{thm:decay} gives constants $b_\delta,B_\delta>0$ such that
\begin{equation}\label{eq:smallarcdecay}
 e^{-B_\delta/(p-1)}
 \le\underline D_{\n}u_{2,\delta,p}(x_0)
 \le\overline D_{\n}u_{2,\delta,p}(x_0)
 \le e^{-b_\delta/(p-1)}.
\end{equation}
For each fixed $\delta$, both estimates hold for $p$ sufficiently close to $1$;
the constants and the range of $p$ may depend on $\delta$.

\subsection{Related results and organization}
The Hopf boundary point lemma was proved by Hopf in \cite{Hopfpaper} for linear
elliptic equations. Tolksdorf \cite[Proposition 3.2.1]{TolksdorfHopfLemma}
established the version for $p$-harmonic functions. Generalizations to degenerate
nonlinear equations appear in \cite[Theorem 5.1.1]{bookMaximumprinciple}; fractional
$p$-Laplacian analogues are studied in \cite{fracplaplacian2017,fracplaplacian2023}.
Interior $C^{1,\alpha}$ regularity was developed in
\cite{DibenedettointeriorC1alpha,EvansinteriorC1alpha,LewisinteriorC1alpha,TolksdorfinteriorC1alpha,UhlenbeckinteriorC1alpha,Ural'cevainteriorC1alpha}, and
boundary regularity in \cite{Liebermanboundaryestimate}.

The variational limit $p\downarrow1$ connects our question with the theory
of least-gradient functions.
Sternberg, Williams, and Ziemer \cite{SWZ} establish existence and uniqueness
of continuous least-gradient extensions under geometric conditions on the domain.
Requiring a minimizer to attain prescribed boundary values is a substantial
restriction: even for continuous boundary values, a least-gradient extension
in $BV(\Omega)\cap C(\overline\Omega)$ need not exist on an arbitrary bounded
domain.
Juutinen \cite{Juutinen} proves uniform convergence of $p$-harmonic extensions
for continuous boundary values on smooth domains with positive mean curvature.
He also proves subsequential convergence in $L^q_{\mathrm{loc}}$, for
$1\le q<d/(d-1)$, to a locally least-gradient function in $BV_{\mathrm{loc}}$
under substantially weaker assumptions, without asserting boundary attainment.
This variational theory involves geometric comparisons of superlevel sets,
whose perimeters are minimized under admissible perturbations.
Our results quantify a different aspect of the singular limit: the asymptotic
scale of the normal Dini derivatives at a fixed boundary point, including the effect
of changing distant boundary values. We do not deduce these derivative estimates
by differentiating a limiting least-gradient function.

Section~\ref{sec:comparison} proves
Theorems~\ref{thm:explosion} and~\ref{thm:decay} using radial barriers.
The quantitative cylinder estimate in Lemma~\ref{lem:cylinder} illustrates
the lower-bound argument and records explicit dependence on the geometry,
$p$, and the dimension; it is not needed in the later proofs.
Section~\ref{sec:critical} proves Theorem~\ref{thm:critical}
using quadratic barriers.

\section{Nonlocal behavior via comparison}\label{sec:comparison}
We first recall the Perron framework and then prove
Theorems~\ref{thm:explosion} and~\ref{thm:decay}.

\subsection{Notation and boundary values}
A domain is a nonempty connected open set. We write $B^m(z,r)$ for an open
Euclidean ball in $\mathbb R^m$ and omit the superscript when the dimension is
clear. We distinguish the interior equation from the interpretation of
boundary values. A \emph{local weak solution} of $\Delta_pu=0$ in $\Omega$
is a function $u\in W^{1,p}_{\mathrm{loc}}(\Omega)$ satisfying
\[
 \int_\Omega |\nabla u|^{p-2}\nabla u\cdot\nabla\varphi\,dx=0
 \quad\text{for all }\varphi\in C_c^\infty(\Omega).
\]
A \emph{$p$-harmonic function} is the continuous representative of such a
solution. This is an interior notion and imposes no boundary condition.

The \emph{weak (variational) Dirichlet problem} additionally prescribes
boundary values in the Sobolev sense: given $g\in W^{1,p}(\Omega)$, one
requires $u\in W^{1,p}(\Omega)$, $u-g\in W^{1,p}_0(\Omega)$, and the weak
equation above. Equivalently, $u$ minimizes
$\int_\Omega|\nabla v|^p\,dx$ over $g+W^{1,p}_0(\Omega)$.
On a Lipschitz domain this means
$u|_{\partial\Omega}=g|_{\partial\Omega}=F$, where restrictions are
understood in the Sobolev sense and equality holds almost everywhere on
$\partial\Omega$ with respect to surface measure.
Our indicator boundary values are interpreted by the Perron method below;
we do not assume that they admit such a Sobolev extension.

A function $v:\Omega\to(-\infty,\infty]$ is $p$-superharmonic
if it is lower semicontinuous, is not identically infinite, and satisfies
the following comparison property: whenever $G\Subset\Omega$ is a subdomain
and $h\in C(\overline G)$ is $p$-harmonic in $G$, the inequality $h\le v$
on $\partial G$ implies $h\le v$ in $G$.
A function is $p$-subharmonic if its negative is $p$-superharmonic.
We use the corresponding comparison principles throughout; see
\cite{HKMbooknonlinearpotentialtheory,Lindqvistbook}.
For a bounded domain $\Omega$ and bounded $F:\partial\Omega\to\mathbb R$,
let $\mathcal U_F$ consist of the
$p$-superharmonic functions bounded below such that
\[
 \liminf_{\Omega\ni x\to z}v(x)\ge F(z)\quad(z\in\partial\Omega).
\]
Define $\mathcal L_F$ analogously, with $p$-subharmonic functions bounded above
and the reverse limsup inequality. The upper and lower Perron solutions are
\[
 \overline H_F(x)=\inf_{v\in\mathcal U_F}v(x),\qquad
 \underline H_F(x)=\sup_{v\in\mathcal L_F}v(x).
\]
When they agree, $F$ is called resolutive and we write $H_F$ for the common
\emph{Perron solution}. For bounded $F$, both Perron solutions are finite
and $p$-harmonic in $\Omega$. In particular, $H_F$ is a local weak solution,
but the Perron construction does not require membership in $W^{1,p}(\Omega)$
or a Sobolev restriction equal to $F$. If
$g\in W^{1,p}(\Omega)\cap C(\overline\Omega)$ and
$F=g|_{\partial\Omega}$, the Perron solution agrees with the variational
Dirichlet solution described above; see
\cite{HKMbooknonlinearpotentialtheory,Lindqvistbook}.

Bounded Lipschitz domains satisfy an exterior cone condition and are
regular for the Dirichlet problem. On a regular domain every bounded upper or
lower semicontinuous boundary function is resolutive
\cite[Proposition 9.31]{HKMbooknonlinearpotentialtheory}. Moreover, $H_F$ tends to $F(z)$ at every
continuity point $z$ of $F$. In particular, indicators of closed boundary
sets are resolutive, and their Perron solutions satisfy $0\le H_F\le1$.

Thus, in our boundary-value problems, the $p$-harmonic function with boundary
values $F$ means $H_F$.

We use boundary comparisons in terms of these limits, rather than assuming
pointwise boundary attainment at jumps of $F$. In the comparisons below,
a barrier either meets a portion where $F$ is locally constant or is bounded by
$0$ or $1$ at the relevant exceptional boundary points.

We use the Dini derivatives defined in \eqref{eq:normaldini}, with $u(x_0)=0$.
If $v\leq u\leq w$ along the inward normal near $x_0$, all three functions
vanish at $x_0$, and $v,w$ are differentiable there, then
\[
 \partial_{\n}v(x_0)\leq\underline D_{\n}u(x_0)
 \leq\overline D_{\n}u(x_0)\leq\partial_{\n}w(x_0).
\]
Indeed, divide the comparison by $t>0$ and take the lower and upper limits
as $t\downarrow0$. No boundary differentiability of $u$ is needed.

\subsection{Algebraic upper and lower bounds}
\begin{proof}[Proof of Theorem~\ref{thm:explosion}]
Translate and rotate so that $x_0=0$ and $\n=e_1$, and put
$\gamma=(d-p)/(p-1)>0$. We may restrict throughout to $1<p<2$.

For part~(b), take $R$ from the hypothesis and choose $M>2\diam(\Omega)$.
On the annulus
\[
 \Omega'=B(-Re_1,R+M)\setminus\overline B(-Re_1,R)
\]
consider the radial $p$-harmonic function
\[
 v(x)=\frac{|x+Re_1|^{-\gamma}-R^{-\gamma}}
 {(R+M)^{-\gamma}-R^{-\gamma}}.
\]
It equals $0$ and $1$ on the inner and outer spheres; see
Figure~\ref{fig:lowerbarrier}. On the part of
$\partial\Omega$ strictly outside the inner sphere, the geometric hypothesis
forces $F=1$ in a boundary neighborhood. On the inner sphere $v=0\le u_p$.
The outer sphere does not meet $\overline\Omega$. Comparison on
$\Omega\cap\Omega'$ therefore gives $v\le u_p$. Dividing along the inward
normal by $t>0$ and taking the lower limit at $0$ yields
\[
 \underline D_{\n}u_p(0)\ge\partial_{e_1}v(0)
 =\frac{\gamma}{R\{1-(R/(R+M))^\gamma\}}
 \ge\frac{d-1}{2R(p-1)}
\]
for $p$ sufficiently close to $1$.
\begin{figure}[tbp]
\centering\includegraphics[width=.83\linewidth]{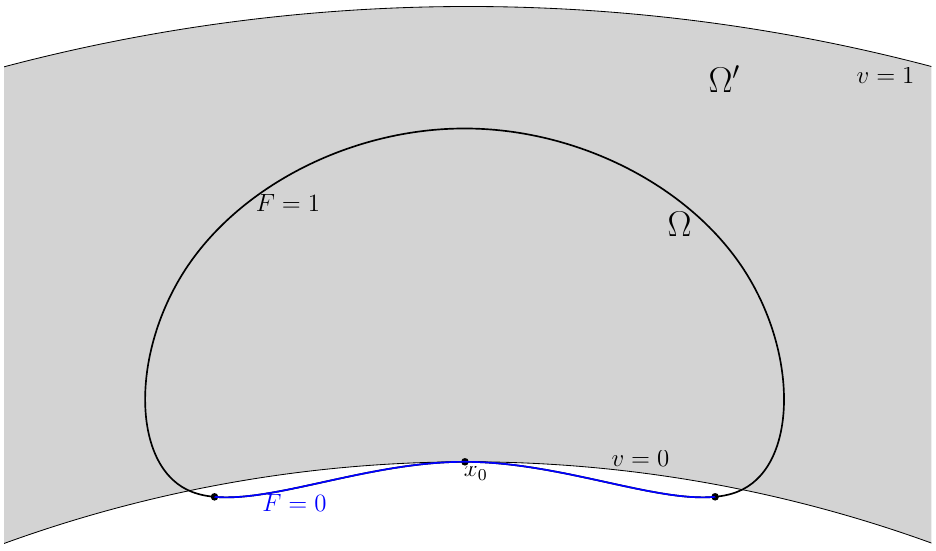}
\caption{The radial comparison function for the lower bound in
Theorem~\ref{thm:explosion}(b).}
\label{fig:lowerbarrier}
\end{figure}

For part~(a), the local $C^2$ boundary admits a sufficiently small exterior
tangent ball with
\[
 \overline B(-\delta e_1,\delta)\cap\overline\Omega=\{0\}.
\]
Since $F$ vanishes near $0$, choose $\delta'>0$ so small that
\[
 \partial\Omega\cap\overline B(-\delta e_1,\delta+\delta')
 \subset\{F=0\}.
\]
Let $\Omega''=\{x:\delta<|x+\delta e_1|<\delta+\delta'\}$ and let $w$ be the
radial $p$-harmonic function on $\Omega''$, with values $0,1$ on its
inner and outer spheres; see Figure~\ref{fig:upperbarrier}.
\begin{figure}[H]
\centering\includegraphics[width=.83\linewidth]{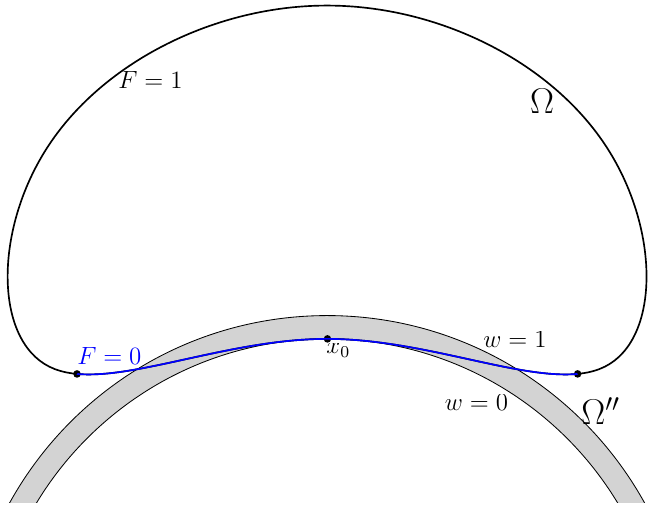}
\caption{The local exterior-annulus barrier for the upper bound in
Theorem~\ref{thm:explosion}(a).}
\label{fig:upperbarrier}
\end{figure}

On the boundary of the intersection with $\Omega$, we have $u_p\le w$:
the domain boundary has value $0$, and the artificial outer sphere has $w=1$.
Thus
\[
 \overline D_{\n}u_p(0)\le\partial_{e_1}w(0)
 =\frac{\gamma}{\delta\{1-(\delta/(\delta+\delta'))^\gamma\}}
 \le\frac{C}{p-1}
\]
for $p$ sufficiently close to $1$. Here $C$ and the threshold depend only on
the indicated geometry, and hence may be written as $C(\Omega,F,x_0)$ and
$p_0(\Omega,F,x_0)$.
\end{proof}

\subsection{Exponential decay under hyperplane separation}
\begin{proof}[Proof of Theorem~\ref{thm:decay}(a)]
The case $F\equiv0$ is immediate. Otherwise, after a rotation and translation,
the hypothesis reads $\{F=1\}\subset\{x_1>0\}$ and $x_{0,1}<0$.
The support of $F$ is compact. A sufficiently large ball, centered far in the
positive $e_1$ direction, therefore contains $\{F=1\}$ and excludes $x_0$
from its closure. Choose such a ball $B(z,R)$, and then choose $\eta>0$ and
$R'>R+\eta$ with
\[
 x_0\notin\overline B(z,R+\eta),\qquad
 \overline\Omega\subset B(z,R').
\]
On $B(z,R')\setminus\overline B(z,R)$ set
\[
 v(x)=\frac{|x-z|^{-\gamma}-(R')^{-\gamma}}
 {R^{-\gamma}-(R')^{-\gamma}}.
\]
The function equals $1$ on the inner sphere and is nonnegative elsewhere in
the annulus. Every point of $\partial\Omega$ outside the inner closed ball
has $F=0$ in a boundary neighborhood. Hence comparison gives $u_p\le v$ on
the intersection of the annulus with $\Omega$. In particular,
\begin{equation}\label{eq:awayfromsupport}
 u_p(x)\le A_p:=
 \frac{(R+\eta)^{-\gamma}-(R')^{-\gamma}}
 {R^{-\gamma}-(R')^{-\gamma}}
 \le2\left(\frac{R}{R+\eta}\right)^\gamma
 \quad\text{for }x\in\Omega\setminus B(z,R+\eta),
\end{equation}
where the last inequality holds for $p$ sufficiently close to $1$.
Apply the exterior-annulus construction from the preceding proof at $x_0$,
choosing its entire intersection with $\Omega$ outside
$\overline B(z,R+\eta)$. The outer boundary value can now be taken to be
$A_p$, so
\[
 \overline D_{\n}u_p(x_0)\le\frac{C A_p}{p-1}
 \le\frac{2C}{p-1}\exp\left(-\frac{d-p}{p-1}
                         \log(1+\eta/R)\right).
\]
The algebraic prefactor is absorbed into the exponential after decreasing
$p_0$. For example, any fixed
$0<c<\tfrac12(d-1)\log(1+\eta/R)$ is admissible for all sufficiently small
$p-1$. This proves \eqref{eq:exponentialupper}.
\end{proof}

\subsection{An exponential lower bound by radial comparison}
\label{sec:lower}

The exponential lower bound follows from explicit radial barriers and a
finite chain of interior balls. We first record the propagation estimate,
then illustrate the argument on a cylinder before proving
Theorem~\ref{thm:decay}(b) on a general domain.
Throughout this subsection, $d\geq2$, $1<p<2$, and
\[
 \gamma=\frac{d-p}{p-1},
 \qquad q=\frac{1}{2^\gamma+1}.
\]
In particular, $\gamma>0$ and $q\geq2^{-(\gamma+1)}$.

\begin{lemma}[Propagation through a ball]\label{lem:propagation}
Let $v$ be a nonnegative $p$-harmonic function in an open set $D$.
Suppose $B(z,4a)\subset D$ and $v\geq m\geq0$ on
$\overline B(z,a)$. Then
\[
 v\geq qm\quad\hbox{on }\overline B(z,2a).
\]
Consequently, if $|z'-z|\leq a$, then $v\geq qm$ on
$\overline B(z',a)$.
\end{lemma}

\begin{proof}
The function
\[
 W(x)=\frac{|x-z|^{-\gamma}-(4a)^{-\gamma}}
 {a^{-\gamma}-(4a)^{-\gamma}}
 \qquad (a<|x-z|<4a)
\]
is $p$-harmonic, equals $1$ on the inner sphere and $0$ on the
outer sphere, and decreases with $|x-z|$.
Comparison gives $v\geq mW$ in the annulus. The comparison is valid
even if its outer sphere meets $\partial D$, since $v\geq0$.
At radius $2a$,
\[
 W=\frac{2^{-\gamma}-4^{-\gamma}}{1-4^{-\gamma}}
   =\frac{1}{2^\gamma+1}=q.
\]
\end{proof}

The next lemma illustrates the main idea: a radial barrier gives a lower
bound near the part of the boundary where the values are positive, and a
chain of interior balls propagates this bound to the point of interest.
It also provides a useful estimate in its own right, with an explicit
geometric constant independent of $d$ and $p$. The lemma is not used in the
later proofs; the general-domain argument below applies
Lemma~\ref{lem:propagation} directly.

\begin{lemma}[Cylinder lower bound]\label{lem:cylinder}
Fix $0<r<R$ and $0<h<H$, and set
\[
 Q=B^{d-1}(0,R)\times(0,H),\qquad
 a=\frac{\min\{R-r,h,H-h\}}{32},\qquad
 C=\left(\left\lceil\frac{H}{a}\right\rceil+1\right)\log2.
\]
If $v$ is nonnegative and $p$-harmonic in $Q$ and
\[
 \liminf_{Q\ni z\to(x,H)}v(z)\geq1
 \qquad\hbox{for every }|x|<R,
\]
then, for every $|x_0|\leq r$,
\[
 v(x_0,h)\geq
 \exp\left(-\frac{C(d-1)}{p-1}\right)
 \geq\exp\left(-\frac{Cd}{p-1}\right).
\]
The hypothesis involves only the open top face. In particular, no
boundary limit at its rim is required for the estimate.
\end{lemma}

\begin{figure}[tbp]
\centering\includegraphics[height=135pt]{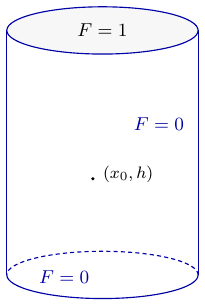}
\caption{The model boundary values for the cylinder lower bound:
value $1$ on the open top face and $0$ on its rim, the sides, and the bottom.
Lemma~\ref{lem:cylinder} only uses nonnegativity and the lower boundary
limit on the open top face.}
\label{fig:cylinderlower}
\end{figure}

\begin{proof}
Fix $|x_0|\leq r$ and put
$z_+=(x_0,H+5a)$ and $z_0=(x_0,H-4a)$.
The ball $B(z_+,5a)$ lies outside $Q$, while $B(z_+,20a)$
meets neither the bottom nor the sides of $Q$. On
$D=Q\cap B(z_+,20a)$ consider the radial function
\[
 W(z)=\frac{|z-z_+|^{-\gamma}-(20a)^{-\gamma}}
 {(5a)^{-\gamma}-(20a)^{-\gamma}}.
\]
It is $p$-harmonic in $D$, lies between $0$ and $1$, and vanishes
on the spherical part of $\partial D$. On the remaining boundary,
which lies strictly inside the top face, the assumed lower boundary
limit of $v$ is at least $1$. Thus comparison gives $v\geq W$ in
$D$. For $z\in\overline B(z_0,a)$ we have $|z-z_+|\leq10a$,
so $v(z)\geq q$.

Join $z_0$ to $(x_0,h)$ by the vertical segment and divide it into
$N=\lceil(H-4a-h)/a\rceil$ equal parts, with successive endpoints
$z_0,z_1,\ldots,z_N=(x_0,h)$. Their distances are at most $a$,
and $B(z_j,4a)\subset Q$ for every $j$.
Applying Lemma~\ref{lem:propagation} successively yields
\[
 v(x_0,h)\geq q^{N+1}
 \geq 2^{-(\gamma+1)(N+1)}\geq e^{-C(\gamma+1)},
\]
as claimed. All comparisons take place away from the top rim.
\end{proof}

\begin{proof}[Proof of Theorem~\ref{thm:decay}(b)]
Write $u=u_p$. We first obtain a lower bound on one interior ball near the portion
of the boundary where $F=1$. Choose $y$ in the relative interior of
$\{F=1\}$ and then $r>0$ so small that
\[
 B(y,4r)\cap\partial\Omega\subset\{F=1\}
\]
and the boundary is a Lipschitz graph in this neighborhood.
Choose $z_-\in B(y,r)\setminus\overline\Omega$ and $0<b<r$
such that $\overline B(z_-,b)\cap\overline\Omega=\varnothing$.
Such an exterior ball exists by the local Lipschitz representation.
There also exist $z_*\in\Omega\cap B(y,r)$ and $s>0$ such that
\[
 \overline B(z_*,s)\subset\Omega\cap B(z_-,2r).
\]
Define
\[
 W(z)=\frac{|z-z_-|^{-\gamma}-(3r)^{-\gamma}}
 {b^{-\gamma}-(3r)^{-\gamma}}
 \quad\hbox{on }\Omega\cap B(z_-,3r).
\]
The function $W$ is $p$-harmonic and takes values in $[0,1]$.
On the part of this domain's boundary contained in $\partial\Omega$,
the Perron solution $u$ tends to $1$: these points lie in the open
patch where $F$ is constant, and Lipschitz boundary points are
regular. On the artificial spherical boundary, $W=0\leq u$.
Comparison therefore gives $u\geq W$. The mean value theorem gives, on
$\overline B(z_*,s)$,
\begin{equation}\label{eq:boundary-seed}
 u\geq
 \frac{(2r)^{-\gamma}-(3r)^{-\gamma}}
 {b^{-\gamma}-(3r)^{-\gamma}}
 \geq\frac13\left(\frac{b}{3r}\right)^{\gamma+1}
 \geq\left(\frac{b}{9r}\right)^{\gamma+1}.
\end{equation}
This comparison does not encounter any discontinuity of $F$.

Next choose an interior tangent ball $B(z_0,R_0)\subset\Omega$
at $x_0$, with $z_0=x_0+R_0\n(x_0)$ and
$\overline B(z_0,R_0)\cap\partial\Omega=\{x_0\}$.
The local $C^2$ assumption permits this choice.
Since $\Omega$ is connected, a polygonal path $\Gamma\subset\Omega$
joins $z_*$ to $z_0$. Its distance $\delta$ from $\partial\Omega$
is positive. Fix
\[
 0<\rho\leq\min\{s,R_0/4,\delta/8\}
\]
and divide $\Gamma$ into finitely many successive pieces of length
at most $\rho$, with endpoints $z_*,z_1,\ldots,z_M=z_0$.
Every ball $B(z_j,4\rho)$ lies in $\Omega$.
Combining \eqref{eq:boundary-seed} with
Lemma~\ref{lem:propagation} gives
\begin{equation}\label{eq:interior-seed}
 u\geq m:=\exp\left[-(\gamma+1)\left(\log\frac{9r}{b}+M\log2\right)\right]
 \quad\hbox{on }\overline B(z_0,\rho).
\end{equation}
All the geometric choices above are independent of $p$; they may
depend on $\Omega$, $F$, and $x_0$.

Finally, compare $u$ on
$B(z_0,R_0)\setminus\overline B(z_0,\rho)$ with
\[
 V(z)=m\,
 \frac{|z-z_0|^{-\gamma}-R_0^{-\gamma}}
 {\rho^{-\gamma}-R_0^{-\gamma}}.
\]
On the inner sphere \eqref{eq:interior-seed} applies, and on the
outer sphere $V=0\leq u$. Thus $u\geq V$.
Since $F$ vanishes near $x_0$, we have $u(x_0)=V(x_0)=0$.
Dividing the last comparison along the inward normal by $t>0$ and taking
the lower limit as $t\downarrow0$, we obtain, by the mean value theorem,
\begin{align*}
 \underline D_{\n}u(x_0)
 &\geq\frac{m\gamma R_0^{-\gamma-1}}{\rho^{-\gamma}-R_0^{-\gamma}}
 \geq\frac{m}{R_0}\left(\frac{\rho}{R_0}\right)^{\gamma+1}\\
 &=\frac{1}{R_0}\exp\left[-\frac{d-1}{p-1}
 \left(\log\frac{9r}{b}+M\log2+\log\frac{R_0}{\rho}\right)\right].
\end{align*}
Since $p-1\leq1$, the fixed factor $R_0^{-1}$ can be absorbed into
the exponential. This proves
\[
 \underline D_{\n}u(x_0)\geq
 \exp\left(-\frac{C(\Omega,F,x_0)}{p-1}\right)
\]
for a positive constant independent of $p$.
\end{proof}

\section{A critical example}\label{sec:critical}

The geometric hypotheses in Theorems~\ref{thm:explosion}(b) and
\ref{thm:decay}(a) need not hold. In this
section we prove Theorem~\ref{thm:critical}: on a cylinder, the inward normal
derivative at the center of the bottom face has order $(p-1)^{-1/2}$.

Write $z=(x,y)\in\mathbb R^{d+1}$, where $x\in\mathbb R^d$, and let
\[
 Q=B^d(0,1)\times(0,1),\qquad
 F(x,y)=
 \begin{cases}
  1,&y=1\ \text{or}\ |x|=1,\\
  0,&y=0\ \text{and}\ |x|<1.
 \end{cases}
\]
Let $u$ be the Perron solution with boundary function $F$. Thus $0\leq u\leq1$,
and $u$ attains $F$ at every boundary point except possibly the bottom rim,
where $F$ is discontinuous. All comparisons below hold in terms of boundary
limits, including at this rim.

\begin{theorem}\label{thm:cylinderprofile}
Let $d\geq1$ and
\[
 1<p<p_0(d):=\min\{2,1+d/400\}.
\]
Then
\begin{equation}\label{eq:cylinderprofile}
 \frac1{80}\sqrt{\frac d{p-1}}\,y
 \leq u(0,y)\leq2\sqrt{\frac d{p-1}}\,y,
 \qquad 0\leq y\leq\sqrt{\frac{p-1}{d}}.
\end{equation}
\end{theorem}

\begin{corollary}\label{cor:criticalderivative}
Under the assumptions of Theorem~\ref{thm:cylinderprofile}, the inward normal
derivative at $(0,0)$ exists and satisfies
\begin{equation}\label{eq:criticalderivativecor}
 \frac1{80}\sqrt{\frac d{p-1}}
 \leq\frac{\partial u}{\partial y}(0,0)
 \leq2\sqrt{\frac d{p-1}}.
\end{equation}
\end{corollary}

\begin{proof}
We first show that
\[
 u(0,y)=\inf_{x\in B^d(0,1)}u(x,y),\qquad 0<y<1.
\]
To justify comparison at the bottom rim, let $F_j$ equal $1$ on the top and
sides and set
\[
 F_j(x,0)=\max\{0,1-j(1-|x|)\},\qquad |x|\leq1.
\]
These continuous boundary functions decrease to $F$. Their Perron solutions
$u_j=H_{F_j}$ are continuous on $\overline Q$ and decrease to $u$.
Indeed, $u_j\geq u$, whereas for each $v\in\mathcal U_F$ and $\varepsilon>0$,
the boundary limits defining $\mathcal U_F$, continuity of $F_j$, and a finite
cover of $\partial Q$ give $v+\varepsilon\in\mathcal U_{F_j}$ for all
sufficiently large $j$. Taking the infimum over $v$ and then letting
$\varepsilon\downarrow0$ proves the assertion.

Fix $a\in B^d(0,1)\setminus\{0\}$ and compare $u_j(x,y)$ with
$u_j(x+a,y)$ on
\[
 Q_a=\{(x,y):|x+a|<1,\ (x+a/2)\cdot a>0,\ 0<y<1\}.
\]
Here $|x|<|x+a|$, so both functions are defined. On the bottom, the required
inequality follows because $F_j$ is radially nondecreasing; on the top both
values are $1$; and on the curved boundary the translated function equals
$1\geq u_j$. On the bisector $(x+a/2)\cdot a=0$, we have $|x|=|x+a|$,
so the values agree by rotational symmetry. Continuity covers the intersections
of these boundary pieces. Comparison gives $u_j(0,y)\leq u_j(a,y)$, and
letting $j\to\infty$ proves the claimed minimum property.

For $0<y_1<1$, the affine $p$-harmonic function
\[
 \ell(x,y)=\frac{y}{y_1}u(0,y_1)
\]
lies below $u$ on the top of $B^d(0,1)\times(0,y_1)$ by the minimum
property, on the sides since $\ell\leq1$, and on the bottom since
$\ell=0\leq u$. The latter comparison also holds at the bottom rim in
terms of boundary limits. Consequently, for $0<y_2<y_1<1$,
\[
 \frac{u(0,y_2)}{y_2}\geq\frac{u(0,y_1)}{y_1}.
\]
Thus $u(0,y)/y$ is nonincreasing in $y$, and \eqref{eq:cylinderprofile}
bounds it near $0$. Its finite limit as $y\downarrow0$ is
$\partial_yu(0,0)$; the same estimate gives the asserted bounds.
\end{proof}

\begin{proof}[Proof of Theorem~\ref{thm:cylinderprofile}]
We use quadratic barriers
\[
 w(x,y)=|x|^2+2ay-by^2.
\]
At points where $\nabla w\ne0$, the sign of $\Delta_pw$ is the sign of
\[
 L_pw:=|\nabla w|^{4-p}\Delta_pw
      =|\nabla w|^2\Delta w+(p-2)\Delta_\infty w,
 \qquad
 \Delta_\infty w:=\sum_{i,j}\partial_iw\,\partial_{ij}w\,\partial_jw.
\]
Indeed,
\[
 \nabla w=2(x,a-by),\qquad
 \Delta w=2(d-b),\qquad
 \Delta_\infty w=8\bigl(|x|^2-b(a-by)^2\bigr),
\]
and consequently
\begin{equation}\label{eq:quadraticoperator}
 L_pw=8\left[|x|^2(d-b+p-2)+(a-by)^2\bigl(d-b(p-1)\bigr)\right].
\end{equation}
The gradients of the barriers used below are nonzero in their respective
comparison domains.

\emph{Upper bound.}
Set
\[
 \delta=\sqrt{\frac{p-1}{d}},\qquad b=\delta^{-2},\qquad a=\delta^{-1}.
\]
Since $1<p<2$, we have $d-b+p-2<0$ and $d-b(p-1)=0$.
Thus $\Delta_pw\leq0$ in the open cylinder
$Q_1=B^d(0,1)\times(0,\delta)$.
On its bottom face, $w=|x|^2\geq0=u$; on its top face,
$w=|x|^2+1\geq1\geq u$; and on its lateral boundary,
$w=1+2y/\delta-(y/\delta)^2\geq1=u$.
At the bottom rim, $w=1\geq\limsup u$, so the jump in $F$ causes no
difficulty. The comparison principle gives $u\leq w$ in $Q_1$, whence
\[
 u(0,y)\leq2y/\delta-(y/\delta)^2\leq2y/\delta,
 \qquad 0\leq y\leq\delta.
\]

\emph{Lower bound.}
Set
\[
 K=\lfloor\delta^{-1}\rfloor\geq20,\qquad y_k=k\delta\quad(k\geq0),
 \qquad Q_k=B^d(0,1)\times(0,y_k)\quad(1\leq k\leq K).
\]
For $y_k<1$, define
\[
 m_k=\inf_{|x|<1}u(x,y_k).
\]
If $y_K=1$, set $m_K=1$, in agreement with the boundary limit on the top
face. The affine function $(x,y)\mapsto y$ is $p$-harmonic and lies below
the boundary limits of $u$, including at the bottom rim. Therefore
\begin{equation}\label{eq:terminallevel}
 u(x,y)\geq y\quad\text{in }Q,\qquad
 m_K\geq y_K\geq1-\delta\geq\tfrac12.
\end{equation}

For the rest of the proof, take
\[
 b=\frac d{2(p-1)}=\frac1{2\delta^2},\qquad
 a_k=b\,y_{k+2}=\frac{k+2}{2\delta}.
\]
For $0<y<y_k$, we have $a_k-by\geq2b\delta=\delta^{-1}$.
Using $d+p-2\geq0$ and $|x|<1$ in~\eqref{eq:quadraticoperator}, we obtain
\[
 \frac18 L_pw
 \geq-b+\frac d2(a_k-by)^2
 \geq-b+\frac d{2\delta^2}\geq0.
\]
Thus the function
\[
 v_k(x,y):=1-f_k(x,y),\qquad
 f_k(x,y):=\frac{b(y-y_{k+2})^2+1-|x|^2}{b\,y_{k+2}^2},
\]
is $p$-subharmonic in $Q_k$: it is a positive affine transformation of
$|x|^2+2a_ky-by^2$. Notice that $y_{k+2}$ is merely a parameter in this
polynomial; no value of $u$ at height $y_{k+2}$ is used.

Since $f_k\geq0$ on $\overline{Q_k}$, we have $v_k\leq1$ there, while
$v_k(x,0)\leq0$. On the top face, $m_kv_k\leq m_k\leq u$;
on the lateral boundary, $m_kv_k\leq1=u$; and on the bottom face,
$m_kv_k\leq0=u$. At the bottom rim $v_k=0\leq\liminf u$.
The same top-face comparison holds if $y_k=1$ by the definition of $m_K$.
Consequently,
\begin{equation}\label{eq:quadraticlower}
 u\geq m_kv_k\qquad\text{in }Q_k.
\end{equation}
For $2\leq k\leq K$ and $|x|<1$, the identity
$y_{k+2}-y_{k-1}=3\delta$ gives
\[
 f_k(x,y_{k-1})
 \leq\frac{9b\delta^2+1}{b(k+2)^2\delta^2}
 =\frac{11}{(k+2)^2}.
\]
Taking the infimum in~\eqref{eq:quadraticlower} therefore yields
\begin{equation}\label{eq:levelrecursion}
 m_{k-1}\geq m_k\left(1-\frac{11}{(k+2)^2}\right),
 \qquad 2\leq k\leq K.
\end{equation}
All factors in this range are positive. Iterating, and using
$\prod_j(1-t_j)\geq1-\sum_jt_j$ for $0\leq t_j\leq1$, gives
\begin{align*}
 m_{20}
 &\geq m_K\prod_{\ell=21}^{K}\left(1-\frac{11}{(\ell+2)^2}\right)\\
 &\geq m_K\left(1-\sum_{\ell=21}^{\infty}
                      \frac{11}{(\ell+1)(\ell+2)}\right)
 =\frac{m_K}{2}\geq\frac14.
\end{align*}
The product is interpreted as $1$ if $K=20$.

Finally, the affine function $m_{20}y/y_{20}$ lies below the boundary
limits of $u$ on $Q_{20}$: this follows from the definition of $m_{20}$
on the top, from $u=1$ on the sides, and from its vanishing on the bottom,
including the rim. Comparison yields
\[
 u(x,y)\geq\frac{m_{20}}{y_{20}}y
 \geq\frac{y}{80\delta},
 \qquad |x|<1,\quad0\leq y\leq y_{20}.
\]
Restricting to $x=0$ and $0\leq y\leq\delta$ completes the proof.
Moreover, comparison with the constant $1/4$ on
$B^d(0,1)\times(y_{20},1)$ gives $u\geq1/4$ above height $y_{20}$.
\end{proof}

\section*{Acknowledgements}
The research of Y. Peres and H. Wang is supported by the National Natural Science
Foundation of China RFIS grant (No. W2531011). In addition, the research of H.
Wang is supported by the National Natural Science Foundation of China
(Grant Nos. 12595284, 12595280).

\bibliographystyle{abbrv}
\bibliography{Hopfref}

\end{document}